\documentclass[11pt,a4paper]{article}

\usepackage[margin=1in]{geometry}
\usepackage[T1]{fontenc}
\usepackage[utf8]{inputenc}
\usepackage{geometry}
\usepackage{lmodern}
\usepackage{microtype}
\usepackage{amsmath,amssymb,amsthm,mathtools}
\usepackage{booktabs,tabularx,array,multirow}
\usepackage{enumitem}
\usepackage{xurl}
\usepackage{xcolor}
\usepackage{hyperref}
\usepackage{fancyhdr}
\usepackage{authblk}
\usepackage{graphicx}
\usepackage[nameinlink,noabbrev]{cleveref}

\allowdisplaybreaks[2]
\numberwithin{equation}{section}

\newtheorem{theorem}{Theorem}[section]
\newtheorem{proposition}[theorem]{Proposition}
\newtheorem{lemma}[theorem]{Lemma}
\newtheorem{corollary}[theorem]{Corollary}
\theoremstyle{definition}
\newtheorem{definition}[theorem]{Definition}
\newtheorem{example}[theorem]{Example}
\theoremstyle{remark}

\newcommand{\I}{\mathbb I}
\newcommand{\Sier}{\mathbb S}
\newcommand{\LFTS}{\mathbf{LFTS}}
\newcommand{\const}[1]{\underline{#1}}
\newcommand{\Fin}{\operatorname{Fin}}
\newcommand{\ev}{\operatorname{ev}}
\newcommand{\sub}{\operatorname{sub}}
\newcommand{\Hom}{\operatorname{Hom}}
\newcommand{\id}{\operatorname{id}}

\title{Exponentiable Objects and Function spaces in Lowen Fuzzy Topological Spaces}
\author{Yongming Li
  \thanks{Corresponding author.
    E-mail:    \href{mailto:liyongm@snnu.edu.cn}
         {\texttt{liyongm@snnu.edu.cn}}}        }
\affil{School of Mathematics and Statistics, Shaanxi Normal University, \\ Xi’an, 710119, China}

\date{}

\begin{document}
\maketitle

\begin{abstract}
We study exponentiable objects and function spaces in the category of stratified Lowen fuzzy topological spaces over \(\I=[0,1]\). Using the Lowen fuzzy Sierpiński object \(\Sier\), which identifies \(\tau_X\) with \(C(X,\Sier)\), we explicitly determine the largest splitting topology on this mapping set. Its open weights \(\Phi:\tau_X\to\I\) are precisely those satisfying Scott continuity and a finite-tier compatibility condition induced by finite powers of \(\Sier\). This yields an intrinsic characterization: \(X\) is exponentiable if and only if every \(\mu\in\tau_X\) satisfies
\[
 \mu=\bigvee_{\lambda\triangleleft\Phi}
       (\const{\Phi(\mu)}\wedge\lambda),
 \qquad
 \lambda\triangleleft\Phi
 \Longleftrightarrow
 \const{\Phi(\nu)}\wedge\lambda\leq\nu
 \quad(\nu\in\tau_X).
\]
When this condition holds, \(Y^X\) has underlying set \(C(X,Y)\), with topology generated by
\([\Phi,v](f)=\Phi(v\circ f)\).
We also obtain a dual closed-set formulation and three applications. Exponentiability implies that \(\tau_X\) is a continuous lattice, although the converse fails. Moreover, a classical space \(X\) is exponentiable exactly when its induced fuzzy space \(\omega X\) is exponentiable in the entire stratified Lowen category. Finally, Lowen compact, strongly fuzzy compact, and \(N\)-compact Hausdorff spaces are exponentiable.

\end{abstract}


\section{Introduction}

Following Zadeh's theory of fuzzy sets, Chang introduced fuzzy
topological spaces by replacing ordinary open subsets with
\([0,1]\)-valued fuzzy subsets~\cite{Chang1968-en}.  Lowen subsequently
emphasized the stratified setting in which all constant fuzzy subsets
are open, and developed fuzzy continuity, initial and final structures,
and several compactness notions~\cite{Lowen1976-en,Lowen1978-en}.
Liu, Zhang, and Luo later studied stratified \(L\)-fuzzy topology and
other fuzzy-topological categories systematically through initially
closed and finally closed full subconstructs~\cite{LiuZhangLuo2000-en}.
Pu and Liu built a neighbourhood and net-convergence theory from fuzzy
points, quasi-coincidence, and \(Q\)-neighbourhoods
~\cite{PuLiu1980-en}.  These developments exposed two persistent
requirements of point-set fuzzy topology: categorical constructions
should extend their classical counterparts, while the interaction of
the truth-value scale \([0,1]\) with products, convergence, and
function spaces must remain visible.

Exponentiable objects lie precisely at this intersection.  In a
category with finite products, an object \(X\) is exponentiable when
the product functor \(-\times X\) has a right adjoint; equivalently,
every target \(Y\) admits a function-space object \(Y^X\) with a
natural bijection
\[
 \Hom(Z\times X,Y)\cong\Hom(Z,Y^X).
\]
The classical category \(\mathbf{Top}\) is not Cartesian closed: its
exponentiable objects are exactly the core-compact spaces, equivalently
those whose open-set lattices are continuous~\cite{EscardoHeckmann-en}.
Determining the exponentiable Lowen fuzzy spaces is therefore not
merely a matter of selecting a convenient topology on one mapping set.
It identifies the part of the category possessing internal Homs, and a
complete answer must construct \emph{every} exponential \(Y^X\), not
only verify evaluation for a special target.

Since the introduction of fuzzy topology, the study of corresponding function spaces has been an important direction, which has also yielded a wealth of results.
Alderton applied the splitting--conjoining framework
directly to fuzzy-topological categories
~\cite{Alderton1989-en}; Peng, Dang--Behera, and J\"ager considered
pointwise convergence, fuzzy compact-open topologies, and continuous
convergence, respectively~\cite{Peng1984-en,DangBehera1996-en,Jager1999-en}.
The work of Lowen and Srivastava on fuzzy Sierpiński objects showed that
fuzzy opens can be classified by continuous maps into a distinguished
object~\cite{Srivastava1984-en,SrivastavaSrivastava1986-en,
LowenSrivastava1988-en,LowenSrivastava1989-en}.  In the broader setting
of \(Q\)-topology, Solovyov and Singh--Srivastava developed
\(Q\)-Sierpiński objects and categorical characterizations
~\cite{Solovyov2008-en,SinghSrivastava2013-en}.  For stratified
\(Q\)-topology, the classifier used in the literature is precisely
$
 \bigl(Q,\langle\{\id_Q\}\cup
 \{\const q:q\in Q\}\rangle\bigr),
$
the abstract form of our \(\Sier\)~\cite{TiwariSrivastava2022-en}.
Tiwari and Srivastava further obtained a Sierpiński-based
splitting--conjoining criterion for exponential \(Q\)-topological
spaces~\cite{TiwariSrivastava2021-en}.  On the other hand, the ``Lowen
spaces'' of Liu and Zhang form the narrower one-step-function
construct~\cite{LiuZhang2000-en}; continuity of the open-set lattice is
sufficient for exponentiability inside that construct
~\cite{LiuZhang2001-en}.  

The abstract function-space criterion for exponentiable objects has a
well-established categorical origin.  Schwarz proved that, in initially
structured categories, exponentiability is equivalent to the existence
of a proper and admissible structure on the set of continuous maps
~\cite{Schwarz1983-en}.  Alderton developed this into the
splitting--conjoining theory for monotopological categories and allowed
the target to be restricted to an initially dense class \cite{Alderton1988-en}. 

Currently, the characterizations of fuzzy exponential objects are all based on category theory, or are defined by means of other fuzzy topological spaces. Researchers have long been seeking an intrinsic characterization of fuzzy exponential objects, analogous to the fact that in general topology, exponential objects correspond precisely to core-compact spaces, which in turn correspond to topological spaces whose open-set lattices are continuous lattices.

In parallel with this, the study of exponential objects in topological molecular lattices is relatively complete. Li has already given a complete intrinsic characterization: a topological molecular lattice is exponential if and only if the dual of its cotopology is a continuous lattice \cite{Li1999-en}. Mirhosseinkhani's work
organizes
products, Isbell-type function spaces, and exponentiable objects in
\(\mathbf{TopFuzz}\), where exponentiability is characterized by core
compactness~\cite{Mirhosseinkhani2024-en}.  These results make clear
that continuous-lattice and core-compactness structures are fundamental
to categorical exponentiation.  Their objects, morphisms, and products,
however, differ from those of the point-set stratified category used
here; their characterizations therefore cannot simply be transferred
to \(\LFTS\).

The remaining obstacle in the point-set Lowen category is twofold.
Ordinary Scott continuity describes directed approximation in the
fuzzy-open lattice but does not control how scalar tiers pass through
finite products; and a splitting function-space topology need not make
evaluation jointly continuous.  By using the Lowen fuzzy Sierpiński
object, this paper makes both obstructions explicit and, within the
ambient stratified Lowen category and its usual continuous maps,
gives an explicit intrinsic characterization of exponentiable objects
and their function-space structures.  The contribution is not the
general splitting--conjoining principle, but the following calculations
and consequences.
\begin{enumerate}[label=(\arabic*)]
 \item We compute the largest splitting Lowen topology \(\Sigma_X\) on
 \(C(X,\Sier)=\tau_X\); its fuzzy opens are exactly the weights that are
 Scott-continuous and satisfy finite-tier compatibility (FTC).
 \item We introduce compatible cores \(\lambda\triangleleft\Phi\) and
 prove that \(X\) is exponentiable exactly when stratified Scott--core
 approximation (SKA) holds.  For every target \(Y\), the exponential
 topology is generated uniformly by
 \([\Phi,v](f)=\Phi(v\circ f)\). A closed-lattice dual of
 SKA is also derived.
 \item We prove that exponentiability forces \(\tau_X\) to be a
 continuous lattice and give a two-point Lowen counterexample showing
 that lattice continuity is not sufficient in the ambient stratified
 category.
 \item We construct tier weights from the way-below relation in the
 classical open-set lattice and prove that \(X\) is exponentiable in
 \(\mathbf{Top}\) exactly when \(\omega X\) is exponentiable in
 \(\LFTS\), by directly verifying FTC and SKA rather than transferring
 exponentials formally through a reflection or coreflection.
 \item We clarify that fuzzy Haudorff spaces with Lowen, strong, or
 \(N\)-compactness are exponentiable.
\end{enumerate}

Section~2 fixes the category and introduces the fuzzy Sierpiński
object.  Section~3 treats splitting, conjoining, and the exponential
law.  Section~4 computes the largest splitting topology, and Section~5
proves the main characterization.  Section~6 contains the
continuous-lattice counterexample, the classical-induction
equivalence, and compactness consequences. Section~7 we conclude the paper.

\section{Preliminaries}

Put \(\I=[0,1]\), ordered in the usual way, with arbitrary joins and
finite meets given by pointwise suprema and minima.  The constant fuzzy
set with value \(a\) is denoted by \(\const a\).

\begin{definition}[Stratified Lowen fuzzy topology]
A Lowen fuzzy topology on a set \(X\) is a subset
\(\tau_X\subseteq\I^X\) containing every constant fuzzy set and closed
under arbitrary pointwise joins and finite pointwise meets.  A map
\(f:X\to Y\) is continuous if \(v\circ f\in\tau_X\) for every
\(v\in\tau_Y\).  These spaces and maps form the category \(\LFTS\)
~\cite{Lowen1976-en,LiuZhangLuo2000-en}.
\end{definition}

This is the \(Q=[0,1]\) instance of stratified \(Q\)-topology, but all
arguments below are given directly in the point-set Lowen category.
Here ``Lowen'' means stratified (all constants are open).  A narrower
use of the term, in which one-step functions form a base, is discussed
separately in Section~\ref{sec:questions}.

\begin{lemma}[Rectangles in products]\label{lem:rect-en}
For Lowen spaces \(Z,X\), a map \(w:Z\times X\to\I\) is open in the
categorical product if and only if
\begin{equation}\label{eq:rect-en}
 w(z,x)=\bigvee_{j\in J}\bigl(a_j(z)\wedge b_j(x)\bigr),
 \qquad a_j\in\tau_Z,\quad b_j\in\tau_X.
\end{equation}
\end{lemma}

\begin{proof}
The product topology is generated by the inverse images of fuzzy opens
under the two projections.  A finite meet of such generators combines
into one factor from \(Z\) and one from \(X\); arbitrary joins then give
\eqref{eq:rect-en}.  The reverse inclusion follows from closure of the
product topology.
\end{proof}

\subsection{The Lowen fuzzy Sierpiński object}

Let \(\delta_{\Sier}\) be the least Lowen topology on \(\I\) generated
by \(\id_{\I}\), and put \(\Sier=(\I,\delta_{\Sier})\).  Early fuzzy
Sierpiński constructions and their categorical role appear in
~\cite{Srivastava1984-en,SrivastavaSrivastava1986-en,
LowenSrivastava1988-en}; the corresponding stratified \(Q\)-topological
form is discussed in~\cite{TiwariSrivastava2022-en}.

\begin{proposition}[Classification of fuzzy opens]\label{prop:class-en}
For every Lowen space \(X\) and every map \(u:X\to\I\),
\[
 u\in\tau_X \quad\Longleftrightarrow\quad u:X\to\Sier
 \text{ is continuous}.
\]
Consequently, naturally in \(X\),
\begin{equation}\label{eq:hom-en}
 C(X,\Sier)=\Hom_{\LFTS}(X,\Sier)\cong\tau_X.
\end{equation}
\end{proposition}

\begin{proof}
If \(u\) is continuous, then \(u=\id_\I\circ u\) is fuzzy open.  The
converse follows because every member of \(\delta_{\Sier}\) is obtained
from \(\id_\I\) and constants by arbitrary joins and finite meets.
\end{proof}

\begin{lemma}[Fuzzy opens of \(\Sier^n\)]\label{lem:sier-n-en}
For \(n\geq1\), a map \(p:\I^n\to\I\) is fuzzy open in the product
\(\Sier^n\) if and only if there are constants
\(c_A\in\I\), \(A\subseteq\{1,\ldots,n\}\), such that
\begin{equation}\label{eq:sugeno-en}
 p(r_1,\ldots,r_n)=
 \bigvee_{A\subseteq\{1,\ldots,n\}}
 \left(c_A\wedge\bigwedge_{i\in A}r_i\right),
\end{equation}
where the empty meet is \(1\).  In particular, the unary fuzzy opens
are exactly
\begin{equation}\label{eq:clamp-en}
 p(r)=a\vee(r\wedge b),\qquad 0\leq a\leq b\leq1.
\end{equation}
\end{lemma}

\begin{proof}
The product is generated by the coordinate maps and the constants.
Every finite meet of generators has the form
\(c\wedge\bigwedge_{i\in A}r_i\).  Terms with the same finite index set
can be joined by taking the supremum of their coefficients, which gives
\eqref{eq:sugeno-en}.  The converse is immediate.
\end{proof}

\section{Function spaces, splitting, and exponentiability}

Let \(C(X,Y)\) denote the set of continuous maps from \(X\) to \(Y\).
A Lowen topology \(\eta\) on this set is \emph{splitting} if every
continuous \(H:Z\times X\to Y\) has continuous transpose
\(\widehat H:Z\to(C(X,Y),\eta)\).  It is \emph{conjoining} if the
evaluation map
\[
 \ev:C(X,Y)\times X\to Y,\qquad \ev(f,x)=f(x),
\]
is continuous.

\begin{lemma}[Exponential-law criterion]\label{lem:exp-law-en}
An object \(X\) is exponentiable in \(\LFTS\) if and only if, for every
Lowen space \(Y\), the set \(C(X,Y)\) admits a topology that is both
splitting and conjoining.  That function space is then \(Y^X\).
\end{lemma}

\begin{proof}
The forward implication is the exponential adjunction and its counit.
Conversely, transposition and evaluation are inverse on underlying sets
and are natural once splitting and conjoining hold.  The underlying set
of an exponential (denoted by $(U(Y^X)$ here) is indeed \(C(X,Y)\): every point selection from the
one-point terminal Lowen space is continuous, so
\(U(Y^X)\cong\Hom(1,Y^X)\cong\Hom(X,Y)\).
\end{proof}

The preceding criterion is the concrete form, in \(\LFTS\), of the
proper/admissible and splitting/conjoining principles of Schwarz and
Alderton~\cite{Schwarz1983-en,Alderton1988-en}.  Alderton subsequently
applied this framework to fuzzy function spaces~\cite{Alderton1989-en},
and J\"ager compared it with continuous-convergence and compact-open
constructions~\cite{Jager1999-en}.  What remains here is to calculate
the largest splitting topology and the conjoining condition explicitly.

Splitting topologies are closed under taking the Lowen topology generated
by their union.  Hence \(C(X,Y)\) has a largest splitting topology.  We
now compute it for \(Y=\Sier\).

\section{Tier-compatible Scott weights}

Regard \(\tau_X\) as a complete lattice in the pointwise order.  Directed
sets are nonempty.

\begin{definition}[Scott weight and finite-tier compatibility]
A map \(\Phi:\tau_X\to\I\) is Scott-continuous if it is monotone and
\[
 \Phi(\bigvee D)=\bigvee_{\mu\in D}\Phi(\mu)
\]
for every directed \(D\subseteq\tau_X\).  Such a weight satisfies
\emph{finite-tier compatibility} (FTC) if, for every \(n\geq1\) and
\(\beta_1,\ldots,\beta_n\in\tau_X\), the map
\begin{equation}\label{eq:ftc-en}
 P_{\Phi,\boldsymbol\beta}(r_1,\ldots,r_n)
 =\Phi\left(\bigvee_{i=1}^n
      (\const{r_i}\wedge\beta_i)\right)
\end{equation}
is fuzzy open in \(\Sier^n\).  Let \(\Sigma_X\) be the set of all
Scott-continuous weights satisfying FTC.
\end{definition}

By Lemma~\ref{lem:sier-n-en}, FTC says exactly that every finite-tier
test \eqref{eq:ftc-en} is a finite lattice polynomial of the form
\eqref{eq:sugeno-en}; it does not impose invariance under arbitrary
affine regradings of \([0,1]\).

\begin{proposition}\label{prop:sigma-en}
\(\Sigma_X\subseteq\I^{\tau_X}\) is a Lowen fuzzy topology.
\end{proposition}

\begin{proof}
Constant weights satisfy the conditions.  Arbitrary joins commute with
directed joins and with all finite-tier tests.  If \(\Phi,\Psi\) satisfy
the conditions, directedness gives
\[
 \Phi\left(\bigvee D\right)\wedge \Psi\left(\bigvee D\right)
 =\bigvee_{\mu\in D}(\Phi(\mu)\wedge\Psi(\mu));
\]
hence \(\Phi\wedge\Psi\) is Scott-continuous, and its FTC tests are
finite meets of fuzzy opens.
\end{proof}

\begin{theorem}[Product stability]\label{thm:prod-stable-en}
For a map \(\Phi:\tau_X\to\I\), the following are equivalent:
\begin{enumerate}[label=(\roman*)]
 \item \(\Phi\in\Sigma_X\);
 \item for every Lowen space \(Z\) and every
 \(w\in\tau_{Z\times X}\), the map
 \[
  z\longmapsto\Phi(w_z),\qquad w_z(x)=w(z,x),
 \]
 belongs to \(\tau_Z\).
\end{enumerate}
\end{theorem}

\begin{proof}
Assume (i) and write \(w\) as in \eqref{eq:rect-en}.  For a nonempty
finite \(F\subseteq J\), put
\[
 w_{F,z}=\bigvee_{j\in F}
       (\const{a_j(z)}\wedge b_j).
\]
These finite partial joins form a directed family with supremum \(w_z\).
Scott continuity and FTC give
\[
 \Phi(w_z)=\bigvee_{F\in\Fin_+(J)}\Phi(w_{F,z})\in\tau_Z.
\]
Indeed, for \(F=\{j_1,\ldots,j_n\}\), the corresponding summand is the
composite of \(Z\to\Sier^n\),
\(z\mapsto(a_{j_1}(z),\ldots,a_{j_n}(z))\), with the FTC test of
\(\Phi\) at \((b_{j_1},\ldots,b_{j_n})\).

Conversely, taking \(Z=\Sier^n\) and
\(w(\boldsymbol r,x)=\bigvee_i(r_i\wedge\beta_i(x))\) proves FTC.  To
prove Scott continuity, let \(D\subseteq\tau_X\) be directed and form
\(Z_D=D\sqcup\{\infty\}\).  Generate a Lowen topology on \(Z_D\) by
the constants and
\[
 u_d(e)=1\ (d\leq e\text{ or }e=\infty),
 \qquad u_d(e)=0\text{ otherwise}.
\]
Every fuzzy open \(q\) in this topology is monotone on \(D\) and
satisfies \(q(\infty)=\bigvee_{e\in D}q(e)\).  The product-open map
\[
 W(z,x)=\bigvee_{d\in D}(u_d(z)\wedge d(x))
\]
has sections \(W_e=e\) and \(W_\infty=\bigvee D\).  Applying (ii) to
\(W\) yields
\(\Phi(\bigvee D)=\bigvee_{e\in D}\Phi(e)\).  Applying this identity to
the two-element directed set \(\{\mu,\nu\}\), \(\mu\leq\nu\), also
gives monotonicity.
\end{proof}

\begin{theorem}[Largest splitting topology]\label{thm:max-split-en}
Under \(C(X,\Sier)=\tau_X\), the Lowen topology \(\Sigma_X\) is the
largest splitting topology on \(C(X,\Sier)\).
\end{theorem}

\begin{proof}
Proposition~\ref{prop:sigma-en} and the forward implication of
Theorem~\ref{thm:prod-stable-en} show that \(\Sigma_X\) is splitting.
If \(\eta\) is any splitting topology and \(\Phi\in\eta\), apply
splitting to each fuzzy open \(w:Z\times X\to\Sier\).  The pullback of
\(\Phi\) along its transpose is \(z\mapsto\Phi(w_z)\); the reverse
implication of Theorem~\ref{thm:prod-stable-en} gives
\(\Phi\in\Sigma_X\).  Thus \(\eta\subseteq\Sigma_X\).
\end{proof}

The following example shows that Scott continuity of a weight is not FTC.

\begin{example}[A one-point weight test]
Let \(X=1\).  Then \(\tau_X\cong\I\).  The weight
\(\Phi(t)=t^2\) is monotone and preserves directed suprema, hence is
Scott-continuous.  For \(n=1\), \(\beta_1=1\), its FTC test is
\(P(r)=r^2\).  If \(r^2=a\vee(r\wedge b)\), substitution of \(0\) and
\(1\) forces \(a=0,b=1\), which would give \(r^2=r\), a contradiction.
Thus \(\Phi\notin\Sigma_X\).
\end{example}

This is a weight-level test, not a space-level counterexample: the
one-point object is exponentiable.  Indeed, the identity weight and
\(\lambda=1\) satisfy SKA below.  The example isolates the finite-tier
condition missing from an unrestricted Scott function-space topology.

\section{The stratified Scott--core theorem}

\begin{definition}[Compatible core and SKA]\label{def:ska-en}
For \(\lambda\in\tau_X\) and \(\Phi\in\Sigma_X\), write
\begin{equation}\label{eq:triangle-en}
 \lambda\triangleleft\Phi
 \quad\Longleftrightarrow\quad
 \const{\Phi(\nu)}\wedge\lambda\leq\nu
 \quad\text{for all }\nu\in\tau_X.
\end{equation}
The space \(X\) satisfies stratified Scott--core approximation (SKA) if
\begin{equation}\label{eq:ska-en}
  \mu=\bigvee_{\substack{\Phi\in\Sigma_X,\ \lambda\in\tau_X\\
                         \lambda\triangleleft\Phi}}
       (\const{\Phi(\mu)}\wedge\lambda)
 \qquad(\mu\in\tau_X).
\end{equation}
\end{definition}

\begin{theorem}[Characterization of exponentiable objects]
\label{thm:main-en}
For a Lowen fuzzy topological space \(X\), the following are equivalent:
\begin{enumerate}[label=(\roman*)]
 \item \(X\) is exponentiable in \(\LFTS\);
 \item the evaluation map
 \[
 e:(\tau_X,\Sigma_X)\times X\to\Sier,
 \qquad e(\mu,x)=\mu(x),
 \]
 is continuous;
 \item \(X\) satisfies SKA.
\end{enumerate}
When these conditions hold, \(Y^X\) has underlying set \(C(X,Y)\),
with the Lowen topology generated by the constants and
\begin{equation}\label{eq:exp-gen-en}
 [\Phi,v](f)=\Phi(v\circ f),
 \qquad \Phi\in\Sigma_X,\quad v\in\tau_Y.
\end{equation}
\end{theorem}

\begin{proof}
Assume (i).  The topology of \(\Sier^X\), transported to \(\tau_X\),
is splitting and its evaluation is continuous.  It is contained in the
largest splitting topology \(\Sigma_X\).  Refining the function-space
factor preserves the openness of the evaluation pullbacks, proving (ii).

Assume (ii).  By Lemma~\ref{lem:rect-en}, the fuzzy open underlying
evaluation admits a rectangle decomposition
\[
 e(\mu,x)=\bigvee_{j\in J}
       (\Phi_j(\mu)\wedge\lambda_j(x)),
 \qquad \Phi_j\in\Sigma_X,\quad\lambda_j\in\tau_X.
\]
Every rectangle is bounded above by \(e\); hence
\(\const{\Phi_j(\nu)}\wedge\lambda_j\leq\nu\) for all \(\nu\), so
\(\lambda_j\triangleleft\Phi_j\).  Fixing \(\mu\) gives both inequalities
in \eqref{eq:ska-en} and proves (iii).

Assume (iii), fix \(Y\), and give \(C(X,Y)\) the topology generated by
\eqref{eq:exp-gen-en}.  If \(H:Z\times X\to Y\) is continuous, then for
\(v\in\tau_Y\) the map \(w(z,x)=v(H(z,x))\) is product-open.  Product
stability gives
\[
 [\Phi,v](\widehat H(z))=\Phi(w_z)\in\tau_Z,
\]
so the topology is splitting.  For evaluation, put \(\mu=v\circ f\).
SKA gives
\begin{align*}
 v(\ev(f,x))
 &=\mu(x)\\
 &=\bigvee_{\lambda\triangleleft\Phi}
   (\Phi(\mu)\wedge\lambda(x))\\
 &=\bigvee_{\lambda\triangleleft\Phi}
   ([\Phi,v](f)\wedge\lambda(x)),
\end{align*}
a join of open rectangles.  Evaluation is continuous, and
Lemma~\ref{lem:exp-law-en} supplies the full exponential adjunction.
\end{proof}

\begin{corollary}[It is enough to test \(\Sier\)]\label{cor:sier-test-en}
The space \(X\) is exponentiable if and only if \(C(X,\Sier)\) admits a
Lowen topology that is both splitting and conjoining.  Whenever one
exists, it may be refined to \(\Sigma_X\) without destroying evaluation
continuity.
\end{corollary}

\begin{proof}
Necessity is immediate from \(\Sier^X\).  For sufficiency, embed the
given splitting topology into \(\Sigma_X\), retain evaluation
continuity, and apply Theorem~\ref{thm:main-en}.
\end{proof}

\subsection{Closed-set form and the ``good compactness'' viewpoint}

Let
\[
 \mathcal C_X=\{1-\mu:\mu\in\tau_X\}
\]
be the fuzzy closed-set lattice.  For \(\lambda\in\tau_X\), put
\(K=1-\lambda\).

\begin{proposition}[Closed-core criterion]\label{prop:closed-en}
For \(\lambda\in\tau_X\) and \(\Phi\in\Sigma_X\), the following are
equivalent:
\begin{enumerate}[label=(\roman*)]
 \item \(\lambda\triangleleft\Phi\);
 \item for every \(F\in\mathcal C_X\),
 \begin{equation}\label{eq:closed-compatible-en}
  F\leq K\vee\const{1-\Phi(1-F)},
  \qquad K=1-\lambda.
 \end{equation}
\end{enumerate}
Moreover, SKA is equivalent to
\begin{equation}\label{eq:closed-ska-en}
 F=\bigwedge_{\lambda\triangleleft\Phi}
   \left((1-\lambda)\vee\const{1-\Phi(1-F)}\right)
 \qquad(F\in\mathcal C_X).
\end{equation}
\end{proposition}

\begin{proof}
Set \(F=1-\nu\), \(K=1-\lambda\), and apply the pointwise De Morgan
laws to
\(\const{\Phi(\nu)}\wedge\lambda\leq\nu\).  This gives
\eqref{eq:closed-compatible-en}.  Complementing both sides of SKA gives
\eqref{eq:closed-ska-en}.
\end{proof}

Formula \eqref{eq:closed-compatible-en} allows a closed core \(K\) to
miss a closed set \(F\) by a constant, quantitatively controlled tier
\(1-\Phi(1-F)\).  Formula \eqref{eq:closed-ska-en} requires every fuzzy
closed set to be recovered as the filtered meet of all such
``closed-core plus tier margin'' approximants.  This is the precise
compactness-like condition produced by the exponential law; it is
strictly more informative than continuity of the open lattice.  We do
not identify it in advance with any existing fuzzy compactness or
well-filteredness axiom, since such an identification requires a
separate equivalence proof.

Using the Gödel implication
\[
 a\Rightarrow_G b=
 \begin{cases}1,&a\leq b,\\b,&a>b,\end{cases}
\]
define the degree of inclusion
\(\sub(\lambda,\nu)=\bigwedge_x
(\lambda(x)\Rightarrow_G\nu(x))\).  Then
\[
 \lambda\triangleleft\Phi
 \quad\Longleftrightarrow\quad
 \Phi(\nu)\leq\sub(\lambda,\nu)
 \quad(\nu\in\tau_X).
\]
Thus \(\Phi\) is a uniform lower bound on the degree to which the core
\(\lambda\) is contained in a variable fuzzy open set, a stratified
analogue of compact containment.

\section{Some applications}
\label{sec:questions}

Two conventions called ``Lowen space'' must be kept separate.  The
ambient category \(\LFTS\) in this paper consists of all stratified
fuzzy topological spaces.  In the narrower Lowen construct of
Liu--Zhang, the topology is required to have a base of one-step
functions \(a\wedge\chi_U\)~\cite{LiuZhang2000-en}.  The latter is an important full
subconstruct, but exponentiability inside a subconstruct does not by
itself establish the universal property against every object of the
ambient category.

\subsection{Exponentiability implies lattice continuity}

Write \(\ll\) for the way-below relation of the complete lattice
\(\tau_X\) \cite{Gierz1980}.

\begin{theorem}[A necessary continuous-lattice condition]
\label{thm:continuous-en}
If \(X\) is exponentiable in \(\LFTS\), then \(\tau_X\) is a continuous
lattice:
\[
 \mu=\bigvee\{\rho\in\tau_X:\rho\ll\mu\}
 \qquad(\mu\in\tau_X).
\]
\end{theorem}

\begin{proof}
By Theorem~\ref{thm:main-en}, SKA holds.  Fix
\(\lambda\triangleleft\Phi\), \(\mu\in\tau_X\), and
\(r<\Phi(\mu)\).  We claim that
\begin{equation}\label{eq:way-en}
 \const r\wedge\lambda\ll\mu.
\end{equation}
If \(D\subseteq\tau_X\) is directed and
\(\mu\leq\bigvee D\), Scott continuity gives
\[
 \Phi(\mu)\leq\Phi(\bigvee D)
 =\bigvee_{\nu\in D}\Phi(\nu).
\]
Choose \(\nu\in D\) with \(r<\Phi(\nu)\).  Compatibility yields
\[
 \const r\wedge\lambda
 \leq\const{\Phi(\nu)}\wedge\lambda\leq\nu,
\]
which proves \eqref{eq:way-en}.  Moreover,
\[
 \const{\Phi(\mu)}\wedge\lambda
 =\bigvee_{r<\Phi(\mu)}(\const r\wedge\lambda).
\]
Substitution into SKA expresses \(\mu\) as a join of elements way below
it.
\end{proof}

The converse is false in the ambient stratified category.  The
following two-point example places the counterexample immediately after
the necessary condition.  Lattice continuity records order
approximation, whereas FTC and compatible cores also record the action
of scalar tiers in products.

\begin{example}[A genuine two-point counterexample]
\label{ex:two-en}
Let \(X=\{p,q\}\), identify fuzzy sets with coordinate pairs, and put
\begin{equation}\label{eq:two-top-en}
 \tau_X=\{(0,0)\}\cup
 \{(s,t)\in[0,1]^2:0<s\leq t\leq1\}.
\end{equation}
Then \(\tau_X\) is a continuous lattice, but \(X\) is not
exponentiable in \(\LFTS\).
\end{example}

\begin{proof}
The constants belong to \(\tau_X\), and coordinatewise arbitrary joins
and finite meets preserve \eqref{eq:two-top-en}.  Thus it is a Lowen
topology.

If \(0<u<s\) and \(u\leq v<t\), then
\((u,v)\ll(s,t)\).  Indeed, when a directed family has supremum above
\((s,t)\), one may choose members whose first and second coordinates
exceed \(u\) and \(v\), respectively, and then take a common upper
bound.  Every nonzero \((s,t)\) is the supremum of
\(((1-1/n)s,(1-1/n)t)\), \(n\geq2\), and the bottom element is way below
itself.  Hence \(\tau_X\) is continuous.

Fix \(0<a<1\) and \(\mu=(a,1)\).  Let
\(\Phi\in\Sigma_X\), \(\lambda=(c,d)\in\tau_X\), and
\(\lambda\triangleleft\Phi\).  We show
\begin{equation}\label{eq:bound-en}
 \Phi(\mu)\wedge d\leq a.
\end{equation}
If \(c=0\), then \(\lambda=(0,0)\).  If \(c>a\), compatibility at the
\(p\)-coordinate of \(\mu\) gives
\(\Phi(\mu)\wedge c\leq a\), hence \(\Phi(\mu)\leq a\).

It remains to treat \(0<c\leq a\).  Choose
\(0<\varepsilon<c\) and put \(\kappa=(\varepsilon,1)\).  The map
\[
 h:\Sier\to\tau_X,\qquad
 h(r)=\kappa\vee\const r=(\varepsilon\vee r,1),
\]
is induced by the product-open map
\((r,x)\mapsto r\vee\kappa(x)\).  Product stability says that
\(P(r)=\Phi(h(r))\) is fuzzy open in \(\Sier\), so
\(P(r)=b_0\vee(r\wedge b_1)\).  Choose
\(\varepsilon<r_0<c\).  Core compatibility at the \(p\)-coordinate of
\(h(r_0)=(r_0,1)\) gives
\(P(r_0)\wedge c\leq r_0\), and therefore
\(P(r_0)\leq r_0\).  In particular \(b_0\leq r_0<a\), whence
\[
 \Phi(\mu)=P(a)=b_0\vee(a\wedge b_1)\leq a.
\]
This proves \eqref{eq:bound-en}.  Every term on the right-hand side of
SKA therefore has value at most \(a\) at \(q\), whereas \(\mu(q)=1\).
SKA fails and Theorem~\ref{thm:main-en} shows that \(X\) is not
exponentiable.
\end{proof}

This example is not one-step generated: a nonzero fuzzy set of the form
\((0,r)\) is not open.  Hence it does not conflict with the
Liu--Zhang sufficiency theorem in the narrower one-step Lowen construct.

\subsection{Induced Lowen spaces of classical exponentiable spaces}

For a topological space \(X\), let
\[
 \omega X=(X,\omega(\mathcal OX)),
 \qquad
\omega(\mathcal OX)=\{u:X\to\I:u\text{ is lower semicontinuous}\}.
\]
This topology is generated by the one-step functions
\(a\wedge\chi_U\), \(U\in\mathcal OX\).  Induced fuzzy spaces and
their relation to classical topologies are treated in
~\cite{WangHu1985-en,LiuZhang2000-en}.

A classical space is exponentiable in \(\mathbf{Top}\) exactly when it
is core-compact, equivalently, when \(\mathcal OX\) is a continuous
lattice.  Reflection and coreflection identify the classical and
induced objects categorically, but do not by themselves transfer
exponentiability to the entire \(\LFTS\), since both the parameter and
the target in the exponential law may be arbitrary Lowen fuzzy spaces.
The following theorem settles the ambient question by verifying FTC
and SKA explicitly.

\begin{theorem}[Exponentiability equivalence for induced spaces]
\label{thm:induced-equivalence-en}
For every classical topological space \(X\), the following are
equivalent:
\begin{enumerate}[label=\textup{(\roman*)}]
 \item \(X\) is exponentiable in \(\mathbf{Top}\);
 \item \(X\) is core-compact, equivalently, \(\mathcal OX\) is a
 continuous lattice;
 \item the induced fuzzy space \(\omega X\) satisfies SKA;
 \item \(\omega X\) is exponentiable in the entire ambient stratified
 Lowen category \(\LFTS\).
\end{enumerate}
\end{theorem}

\begin{proof}
The classical exponentiation theorem gives
(i)\(\Leftrightarrow\)(ii), and Theorem~\ref{thm:main-en} gives
(iii)\(\Leftrightarrow\)(iv).  We prove the substantive implication
(ii)\(\Rightarrow\)(iii), and then verify
(iv)\(\Rightarrow\)(ii) directly.

Assume that $X$ is core-compact, equivalently, \(\mathcal OX\) is a
 continuous lattice. We first define tier weights from the way-below relation.

For \(\nu\in\omega(\mathcal OX)\) and \(t\in\I\), put
\[
 \nu_t=\{x\in X:\nu(x)>t\}\in\mathcal OX.
\]
For every nonempty \(U\in\mathcal OX\), define
\begin{equation}\label{eq:induced-weight-en}
 \Phi_U(\nu)
 =\bigvee\{t\in\I:U\ll\nu_t\},
\end{equation}
where the supremum of the empty set is \(0\).

Furthermore, we show that $\Phi_U$ is Scott continuous and satisfies FTC.

Monotonicity of $\Phi_U$ is immediate.  Let \(D\) be a directed family in
\(\omega(\mathcal OX)\) and put \(\nu=\bigvee D\).  Clearly,
\[
 \bigvee_{\rho\in D}\Phi_U(\rho)\leq\Phi_U(\nu).
\]
Conversely, fix \(s<\Phi_U(\nu)\).  There is \(t>s\) such that
\(U\ll\nu_t\).  Moreover,
\[
 \nu_t
 =\left\{x:\bigvee_{\rho\in D}\rho(x)>t\right\}
 =\bigcup_{\rho\in D}\rho_t,
\]
and the family \(\{\rho_t:\rho\in D\}\) is directed.  By interpolation
of the way-below relation in the continuous lattice \(\mathcal OX\),
choose \(W\) with
\[
 U\ll W\ll\nu_t.
\]
Some \(\rho_0\in D\) satisfies \(W\subseteq(\rho_0)_t\).  Hence
\[
 U\ll W\subseteq(\rho_0)_t\subseteq(\rho_0)_s,
\]
so \(\Phi_U(\rho_0)\geq s\).  Letting
\(s\uparrow\Phi_U(\nu)\) proves
\[
 \Phi_U\left(\bigvee D\right)
 =\bigvee_{\rho\in D}\Phi_U(\rho).
\]

For the FTC, fix \(\beta_1,\ldots,\beta_n\in\omega(\mathcal OX)\).  For
\(A\subseteq\{1,\ldots,n\}\), put
\[
 \beta_A=\bigvee_{i\in A}\beta_i,\qquad
 c_A=\Phi_U(\beta_A),
\]
with \(\beta_\varnothing=0\).  Then
\begin{equation}\label{eq:induced-ftc-polynomial-en}
 \Phi_U\left(\bigvee_{i=1}^n
       (\const{r_i}\wedge\beta_i)\right)
 =
 \bigvee_{A\subseteq\{1,\ldots,n\}}
 \left(c_A\wedge\bigwedge_{i\in A}r_i\right).
\end{equation}
To verify the identity, compare strict \(t\)-cuts.  Its left-hand side
is greater than \(t\) exactly when some \(s>t\) satisfies
\[
 U\ll
 \left\{x:\bigvee_i(r_i\wedge\beta_i(x))>s\right\}
 =
 \bigcup_{\{i:r_i>s\}}\{\beta_i>s\}.
\]
Taking \(A=\{i:r_i>s\}\) gives
\(c_A>t\) and \(\bigwedge_{i\in A}r_i>t\).  Conversely, if an \(A\)
satisfies these two strict inequalities, choose
\[
 t<q<c_A\wedge\bigwedge_{i\in A}r_i.
\]
Since \(q<c_A=\Phi_U(\beta_A)\), some \(b>q\) satisfies
\(U\ll(\beta_A)_b\).  Since \(q<r_i\) for every \(i\in A\),
\[
 (\beta_A)_b\subseteq(\beta_A)_q
 \subseteq
 \left\{x:\bigvee_i(r_i\wedge\beta_i(x))>q\right\}.
\]
The left-hand side of \eqref{eq:induced-ftc-polynomial-en} is therefore
at least \(q>t\).  Thus the identity holds.  Its right-hand
side is a finite lattice polynomial of the form in
Lemma~\ref{lem:sier-n-en}; hence
\[
 \Phi_U\in\Sigma_{\omega X}.
\]

For the compatible crisp cores, we claim
\begin{equation}\label{eq:induced-compatible-core-en}
 \chi_U\triangleleft\Phi_U.
\end{equation}
If \(x\notin U\), then
\((\const{\Phi_U(\nu)}\wedge\chi_U)(x)=0\).  If \(x\in U\), then for
every \(t<\Phi_U(\nu)\) there is \(s>t\) with \(U\ll\nu_s\).
Consequently, \(x\in\nu_s\) and \(\nu(x)>s>t\).  Letting
\(t\uparrow\Phi_U(\nu)\) gives \(\Phi_U(\nu)\leq\nu(x)\).  Therefore
\[
 \const{\Phi_U(\nu)}\wedge\chi_U\leq\nu
 \qquad(\nu\in\omega(\mathcal OX)),
\]
which proves \eqref{eq:induced-compatible-core-en}.

Now let us show that $\omega X$ sasifies SKA, and thus complete the proof of (ii)\(\Rightarrow\)(iii).

Fix \(\mu\in\omega(\mathcal OX)\), \(x\in X\), and \(r<\mu(x)\).
Choose \(a\) such that
\[
 r<a<\mu(x).
\]
The open set \(V=\mu_a\) contains \(x\).  Core compactness provides an
open set \(U\) with
\[
 x\in U\ll V.
\]
Equations~\eqref{eq:induced-weight-en} and
\eqref{eq:induced-compatible-core-en} give
\[
 \Phi_U(\mu)\geq a>r,\qquad
 \chi_U\triangleleft\Phi_U,
\]
and hence
\[
 (\const{\Phi_U(\mu)}\wedge\chi_U)(x)
 =\Phi_U(\mu)>r.
\]
Letting \(r\uparrow\mu(x)\) proves that the right-hand side of SKA is
at least \(\mu\) pointwise; the reverse inequality is built into core
compatibility.  Thus \(\omega X\) satisfies SKA.

Next, we verify (iv)\(\Rightarrow\)(ii).

If \(\omega X\) is exponentiable in \(\LFTS\), then
Theorem~\ref{thm:continuous-en} shows that
\(\omega(\mathcal OX)\) is a continuous lattice.  For
\(V\in\mathcal OX\),
\[
 \chi_V
 =\bigvee\{\rho\in\omega(\mathcal OX):\rho\ll\chi_V\}.
\]
For each \(\rho\ll\chi_V\), put
\[
 U_\rho=\{x:\rho(x)>0\}.
\]
If a directed family \(\mathcal D\subseteq\mathcal OX\) satisfies
\(V\subseteq\bigcup\mathcal D\), then
\[
 \chi_V\leq\bigvee_{W\in\mathcal D}\chi_W.
\]
The relation \(\rho\ll\chi_V\) gives a \(W\in\mathcal D\) with
\(\rho\leq\chi_W\), and therefore \(U_\rho\subseteq W\).  Hence
\(U_\rho\ll V\).  Pointwise continuity also gives
\[
 V=\bigcup_{\rho\ll\chi_V}U_\rho.
\]
Thus \(\mathcal OX\) is a continuous lattice, and \(X\) is
core-compact.
\end{proof}

\begin{corollary}[Locally compact Hausdorff induced spaces]
\label{cor:lch-en}
If \(X\) is locally compact Hausdorff, then \(\omega X\) is
exponentiable in the entire \(\LFTS\).
\end{corollary}

\begin{proof}
Every locally compact Hausdorff space is core-compact; apply
Theorem~\ref{thm:induced-equivalence-en}.
\end{proof}

The Liu--Zhang theorem supplies an exponential law only in the
narrower one-step construct.  Theorem~\ref{thm:induced-equivalence-en}
instead verifies FTC, compatible cores, and SKA directly, and therefore
constructs exponentials against all Lowen fuzzy parameter and target
spaces.

\subsection{Fuzzy compact Hausdorff spaces are exponentiable}

For \(x\in X\) and \(0<\alpha\leq1\), the fuzzy point \(x_\alpha\)
has value \(\alpha\) at \(x\) and \(0\) elsewhere.  Fuzzy sets \(A,B\)
are \emph{quasi-coincident}, written \(A\,q\,B\), if
\[
 \exists z\in X,\qquad A(z)+B(z)>1.
\]
A fuzzy set \(N\) is a \(Q\)-neighbourhood of \(x_\alpha\) if some
\(U\in\tau_X\) satisfies \(x_\alpha qU\) and \(U\leq N\).

\begin{definition}[Wang's \(FT_2\) condition]\label{def:ft2-en}
A fuzzy topological space is fuzzy Hausdorff (\(FT_2\)) if, whenever
\(x_\alpha\) and \(y_\beta\) have different supports, they possess open
\(Q\)-neighbourhoods \(U,V\) such that
\begin{equation}\label{eq:ft2-en}
 U\wedge V=0.
\end{equation}
\end{definition}

This is the Pu--Liu \(Q\)-neighbourhood separation used in Wang's
fuzzy-net framework~\cite{PuLiu1980-en,Srivastava1981-en,Wang1983-en}.
Notice that \(x_\alpha qU\) means
\(U(x)>1-\alpha\); the definition therefore quantifies both supports
and all heights.

Wang introduced nice fuzzy compactness, or \(N\)-compactness, by fuzzy
nets~\cite{Wang1983-en}.  Zhao subsequently extended the notion to
general \(L\)-fuzzy topological spaces and gave an equivalent covering
formulation~\cite{Zhao1987-en}.  Since the present paper fixes
\(L=[0,1]\), that formulation takes the following explicit form.

For \(\alpha\in(0,1]\), a family
\(\mathcal U\subseteq\tau_X\) is an \emph{open
\(\alpha\)-\(Q\)-cover} of \(X\) if
\begin{equation}\label{eq:alpha-q-cover-en}
 \forall x\in X\ \exists U\in\mathcal U:\quad x_\alpha qU,
 \qquad\text{equivalently}\qquad
 \bigvee_{U\in\mathcal U}U(x)>1-\alpha\quad(x\in X).
\end{equation}
It is an \(\alpha^{-}\)-\(Q\)-cover if there is one
\(\beta\in(0,\alpha)\) for which it is a \(\beta\)-\(Q\)-cover, that is,
\begin{equation}\label{eq:alpha-minus-q-cover-en}
 \bigvee_{U\in\mathcal U}U(x)>1-\beta\quad(x\in X).
\end{equation}

We use the following standard hierarchy.
\begin{enumerate}[label=(\arabic*)]
 \item \emph{Lowen fuzzy compactness}: for every \(a\in(0,1]\), if
 \(\mathcal U\subseteq\tau_X\) satisfies
 \(\bigvee\mathcal U(x)\geq a\) for all \(x\), then for any $0<b<a$, there is a finite
 \(\mathcal U_0\subseteq\mathcal U\) still satisfies
 \(\bigvee\mathcal U_0(x)\geq b\) for all \(x\).
 ~\cite{Lowen1976-en,Lowen1978-en}.
 \item \emph{Strong fuzzy compactness}: for every \(a\in[0,1)\), if
 \(\mathcal U\subseteq\tau_X\) satisfies
 \(\bigvee\mathcal U(x)>a\) for all \(x\), then some finite
 \(\mathcal U_0\subseteq\mathcal U\) still satisfies
 \(\bigvee\mathcal U_0(x)>a\) for all \(x\).  Equivalently, the space
 is \(a\)-compact for every \(a\); this implies Lowen fuzzy compactness.
 \item \emph{Nice fuzzy compactness}, or \(N\)-compactness: for every
\(\alpha\in(0,1]\), each open \(\alpha\)-\(Q\)-cover contains a finite
subfamily that is an \(\alpha^{-}\)-\(Q\)-cover..
\end{enumerate}
The covering formulation directly proves that \(N\)-compactness
implies strong fuzzy compactness.  Given an open \(a\)-shading as in
item~(3), put \(\alpha=1-a\).  Equation~\eqref{eq:alpha-q-cover-en}
makes it an open \(\alpha\)-\(Q\)-cover.  A finite
\(\alpha^{-}\)-\(Q\)-subcover supplies some \(0<\beta<\alpha\), and
\eqref{eq:alpha-minus-q-cover-en} gives
\(\bigvee\mathcal U_0(x)>1-\beta>1-\alpha=a\) for every \(x\).
Thus
\begin{equation}\label{eq:compact-chain-en}
\begin{aligned}
 N\text{-compact}&\Longrightarrow
 \text{strongly fuzzy compact}
 &\Longrightarrow\text{Lowen fuzzy compact}.
\end{aligned}
\end{equation}
For more relationships concerning fuzzy compactness, we may refer to references \cite{Wang1985-cn,LiuLuo1997}. This chain concerns only the fixed covering definitions above; no
reverse implication is used below.  More importantly, fuzzy Hausdorff
axioms and fuzzy compactness notions are not unique, and different
combinations need not satisfy the same theorem asserting that compact
Hausdorff fuzzy spaces are topologically generated.  The work of
Warner--McLean on compact Hausdorff \(L\)-fuzzy spaces and Shi's later
comparison of fuzzy compactness notions both underscore the need to
check the precise hypotheses~\cite{WarnerMcLean1993-en,Shi2005-en}.

\begin{theorem}[Precise compact--Hausdorff consequence]
\label{thm:compact-ft2-en}
Let \(X\) be a stratified Lowen fuzzy space, use the Hausdorff axiom of
Definition~\ref{def:ft2-en}, the following statements hold.
\begin{enumerate}[label=(\roman*)]
 \item If \(X\) is Lowen fuzzy compact and \(FT_2\), then \(X\) is
 exponentiable in \(\LFTS\).
 \item If \(X\) is strongly fuzzy compact and \(FT_2\), then \(X\) is
 exponentiable in \(\LFTS\).
 \item If \(X\) is \(N\)-compact and \(FT_2\), then \(X\) is exponentiable in
 \(\LFTS\).
\end{enumerate}
\end{theorem}

\begin{proof}
In either case \eqref{eq:compact-chain-en} yields Lowen fuzzy
compactness. Lowen's theorem~\cite{Lowen1981-en} that compact Hausdorff fuzzy topological
spaces are topologically generated then gives a compact Hausdorff
topological space \(\mathcal OX\) with the original fuzzy topology
equal to \(\omega(\mathcal OX)\).  Every compact Hausdorff space is
locally compact, so Corollary~\ref{cor:lch-en} applies in the full ambient
category.
\end{proof}

\section{Conclusion}

In the ambient stratified Lowen category, exponentiability is governed
by two inseparable pieces of structure: admissible weights form the
largest splitting topology \(\Sigma_X\), hence are Scott-continuous and
finite-tier compatible; and the original fuzzy opens must be recovered
from those weights and their compatible cores by SKA.  The closed-set
formula expresses the same condition as approximation by closed cores
with quantitatively controlled tier margins.
The additional results sharpen the boundary.  Exponentiability always
implies continuity of \(\tau_X\), but Example~\ref{ex:two-en} disproves
the converse.  The tier weights constructed from the classical
way-below relation satisfy FTC and have compatible crisp cores, yielding
\[
 X\text{ exponentiable in }\mathbf{Top}
 \quad\Longleftrightarrow\quad
 \omega X\text{ exponentiable in }\LFTS.
\]
Fuzzy compact-Hausdorff topological spaces are induced by compact-Hausdorff topological spaces and thus exponentiable in \(\LFTS\).

In this paper, we introduce the SKA property of fuzzy topological spaces and prove that it characterizes the exponentiability of a topological space. Nevertheless, its further order-theoretic, algebraic, and topological aspects are still worthy of further study.


\begin{thebibliography}{99}

\bibitem{Alderton1988-en}
I. W. Alderton,
Splitting and conjoining objects in monotopological categories,
\emph{Topology and its Applications}
\textbf{29} (1988), 223--235,
doi:10.1016/0166-8641(88)90022-3.

\bibitem{Alderton1989-en}
I. W. Alderton,
Function spaces in fuzzy topology,
\emph{Fuzzy Sets and Systems}
\textbf{32} (1989), 115--124,
doi:10.1016/0165-0114(89)90092-4.

\bibitem{Chang1968-en}
C. L. Chang,
Fuzzy topological spaces,
\emph{Journal of Mathematical Analysis and Applications}
\textbf{24} (1968), 182--190.

\bibitem{DangBehera1996-en}
S. Dang and A. Behera,
On fuzzy compact-open topology,
\emph{Fuzzy Sets and Systems}
\textbf{80} (1996), 377--381.

\bibitem{EscardoHeckmann-en}
M. H. Escard\'o and R. Heckmann,
Topologies on spaces of continuous functions,
\emph{Topology Proceedings}
\textbf{26} (2001--2002), 545--564.

\bibitem{Gierz1980}
G. Gierz et al., A compendium of continuous lattices,
Springer, Berlin, 1980

\bibitem{Jager1999-en}
G. J\"ager,
On fuzzy function spaces,
\emph{International Journal of Mathematics and Mathematical Sciences}
\textbf{22} (1999), 727--737,
doi:10.1155/S0161171299227275.

\bibitem{Li1999-en}
Y. -M. Li,
Exponentiable objects in the category of topological molecular lattices,
\emph{Fuzzy Sets and Systems}
\textbf{104} (1999), 407--414,
doi:10.1016/S0165-0114(98)00316-9.

\bibitem{LiuLuo1997}
Y. -M. Liu, M. -K. Luo, Fuzzy Topology, World Scientific, Singapore, 1997.

\bibitem{LiuZhang2000-en}
Y. -M. Liu and D. -X. Zhang,
Lowen spaces,
\emph{Journal of Mathematical Analysis and Applications}
\textbf{241} (2000), 30--38,
doi:10.1006/jmaa.1999.6586.

\bibitem{LiuZhang2001-en}
Y. -M. Liu and D.-X. Zhang,
On exponential Lowen spaces,
in \emph{Proceedings of the Joint 9th IFSA World Congress and
20th NAFIPS International Conference}, vol.~2, 2001,
doi:10.1109/NAFIPS.2001.944775.

\bibitem{LiuZhangLuo2000-en}
Y. -M. Liu, D. -X. Zhang, and M. -K. Luo,
Initial and final structures of fuzzy topological spaces,
\emph{Journal of Mathematical Analysis and Applications}
\textbf{251} (2000), 649--668,
doi:10.1006/jmaa.2000.7036.

\bibitem{Lowen1976-en}
R. Lowen,
Fuzzy topological spaces and fuzzy compactness,
\emph{Journal of Mathematical Analysis and Applications}
\textbf{56} (1976), 621--633.

\bibitem{Lowen1978-en}
R. Lowen,
A comparison of different compactness notions in fuzzy topological spaces,
\emph{Journal of Mathematical Analysis and Applications}
\textbf{64} (1978), 446--454.

\bibitem{Lowen1981-en}
R. Lowen,
Compact Hausdorff fuzzy topological spaces are topological,
\emph{Topology and its Applications}
\textbf{12} (1981), 65--74,
doi:10.1016/0166-8641(81)90030-4.

\bibitem{LowenSrivastava1988-en}
R. Lowen and A. K. Srivastava,
Sierpinski objects in subcategories of FTS,
\emph{Quaestiones Mathematicae}
\textbf{11} (1988), 181--193.

\bibitem{LowenSrivastava1989-en}
R. Lowen and A. K. Srivastava,
\(\mathrm{FTS}_0\): The epireflective hull of the Sierpinski object in FTS,
\emph{Fuzzy Sets and Systems}
\textbf{29} (1989), 171--176.

\bibitem{Mirhosseinkhani2024-en}
G. Mirhosseinkhani,
Some categorical structures of topological fuzzes,
in P. Bracken (ed.), \emph{Recent Topics on Topology---From Classical
to Modern Applications}, IntechOpen, 2024,
doi:10.5772/intechopen.1005849.

\bibitem{Peng1984-en}
Y. Peng,
Topological structure of a fuzzy function space---the pointwise convergent
topology and compact open topology,
\emph{Kexue Tongbao (English Edition)}
\textbf{29} (1984), 289--292.

\bibitem{PuLiu1980-en}
P. -M. Pu and Y. -M. Liu,
Fuzzy topology I: Neighborhood structure of a fuzzy point and
Moore--Smith convergence,
\emph{Journal of Mathematical Analysis and Applications}
\textbf{76} (1980), 571--599.

\bibitem{Schwarz1983-en}
F. Schwarz,
Powers and exponential objects in initially structured categories and
applications to categories of limit spaces,
\emph{Quaestiones Mathematicae}
\textbf{6} (1983), 227--254,
doi:10.1080/16073606.1983.9632302.

\bibitem{SinghSrivastava2013-en}
A. K. Singh and R. Srivastava,
A characterization of the category \(Q\)-TOP,
\emph{Fuzzy Sets and Systems}
\textbf{227} (2013), 46--50.

\bibitem{Shi2005-en}
F.-G. Shi,
A new notion of fuzzy compactness in \(L\)-topological spaces,
\emph{Information Sciences}
\textbf{173} (2005), 35--48,
doi:10.1016/j.ins.2004.06.004.

\bibitem{Solovyov2008-en}
S. A. Solovyov,
Sobriety and spatiality in varieties of algebras,
\emph{Fuzzy Sets and Systems}
\textbf{159} (2008), 2567--2585,
doi:10.1016/j.fss.2008.02.010.


\bibitem{Srivastava1981-en}
R. Srivastava, S. N. Lal and A. K. Srivastava,
Fuzzy Hausdorff topological spaces,
\emph{Journal of Mathematical Analysis and Applications}
\textbf{81} (1981), 497--506,
doi:10.1016/0022-247X(81)90078-0.

\bibitem{Srivastava1984-en}
A. K. Srivastava,
Fuzzy topology and Sierpinski fuzzy space,
\emph{Journal of Mathematical Analysis and Applications}
\textbf{103} (1984), 103--117,
doi:10.1016/0022-247X(84)90160-4.

\bibitem{SrivastavaSrivastava1986-en}
R. Srivastava and A. K. Srivastava,
The Sierpinski object in fuzzy topology,
in \emph{Modern Analysis and Its Applications},
Prentice-Hall of India, New Delhi, 1986, pp.~17--25.

\bibitem{TiwariSrivastava2022-en}
H. Tiwari and R. Srivastava,
On coreflective hulls in Str-\(Q\)-TOP,
\emph{Soft Computing}
\textbf{26} (2022), 527--534,
doi:10.1007/s00500-021-06537-z.

\bibitem{TiwariSrivastava2021-en}
H. Tiwari and R. Srivastava,
Exponential \(Q\)-topological spaces,
\emph{Fuzzy Sets and Systems}
\textbf{406} (2021), 58--65,
doi:10.1016/j.fss.2019.11.012.

\bibitem{Wang1983-en}
G. -J. Wang,
A new fuzzy compactness defined by fuzzy nets,
\emph{Journal of Mathematical Analysis and Applications}
\textbf{94} (1983), 1--23,
doi:10.1016/0022-247X(83)90002-1.

\bibitem{Wang1985-cn}
G. -J. Wang, Theory of L-Fuzzy Topological Spaces, Shaanxi Normal University Press, Xi'an, 1985.

\bibitem{WangHu1985-en}
G. -P. Wang and L. Hu,
On induced fuzzy topological spaces,
\emph{Journal of Mathematical Analysis and Applications}
\textbf{108} (1985), 495--506.

\bibitem{WarnerMcLean1993-en}
M. W. Warner and M. W. McLean,
On compact Hausdorff \(L\)-fuzzy spaces,
\emph{Fuzzy Sets and Systems}
\textbf{56} (1993), 103--110,
doi:10.1016/0165-0114(93)90190-S.

\bibitem{Zhao1987-en}
D. Zhao,
The \(N\)-compactness in \(L\)-fuzzy topological spaces,
\emph{Journal of Mathematical Analysis and Applications}
\textbf{128} (1987), 64--79,
doi:10.1016/0022-247X(87)90214-9.

\end{thebibliography}
\end{document}